\documentclass[11pt, a4paper]{article}

\usepackage{amsmath, amssymb, amsthm, xspace, a4wide, amscd}
\usepackage[english]{babel}
\usepackage[T1]{fontenc}
\usepackage[latin1]{inputenc}

\newtheorem{thm}{Theorem}
\newtheorem{lemma}{Lemma}

\theoremstyle{definition}
\newtheorem{defin}{Definition}

\theoremstyle{remark}
\newtheorem{example}{Example}
\newtheorem{assumption}{Assumption A\hspace*{-2pt}}
\newtheorem*{remark}{Remark}

\newcommand{\ii}{\ensuremath{\mathbf{1}}}
\newcommand{\rr}{\ensuremath{\mathbb{R}}}

\newcommand{\xx}{{\sf{X}}}
\newcommand{\pp}{{\sf{P}}}
\newcommand{\mm}{{\sf{m}}}
\newcommand{\llf}{{\sf{L}}}
\newcommand{\bb}{\mathcal{B}}
\newcommand{\ff}{\mathcal{F}}

\newcommand{\aaa}{{\sf A}}

\newcommand{\rp}{\stackrel{{\sf P}}{\to}}

\newcommand{\ee}{{\sf E}}

\begin{document}

\begin{center}\Large
Regularity of paths of stochastic measures
\end{center}

\begin{center}
Vadym Radchenko
\footnote{Department of Mathematical Analysis, Taras Shevchenko National University of Kyiv, 
01601 Kyiv, {Ukraine},\\ vadymradchenko@knu.ua}
\end{center}

\begin{abstract}
Random functions $\mu(x)$, generated by values of stochastic measures are considered. The Besov regularity of the continuous paths of $\mu(x)$, $x\in[0,1]^d$ is proved. Fourier series expansion of $\mu(x)$, $x\in[0,2\pi]$ is obtained. These results are proved under weaker conditions than similar results in previous papers.
\end{abstract}

\section{Introduction}\label{scintr}

In this paper, we consider the properties of paths of processes generated by values of stochastic measures (SMs). SM is a stochastic set function that is $\sigma$-additive in probability, see exact definition and examples in subsection~\ref{ss:stochm}.

Besov regularity of paths of SMs was studied in \cite{rads06} and \cite{radt09}, Fourier series expansions of paths of SMs were obtained in~\cite{radt18}. Some important results in these papers were obtained under the following condition for SM $\mu$ defined on measurable space $(\xx, \bb)$.

\begin{assumption}\label{asssq}
There exists a real-valued finite measure~$\mm$
on~$(\xx, \bb)$ with the following property: if a measurable function $g:\xx\to\rr$ be such that $\int_{\xx}g^2\,d\mm<+\infty$ then $g$ is integrable w.~r.~t.~$\mu$ on~$\xx$.
\end{assumption}

In this paper, we prove that some statements from \cite{radt09} and~\cite{radt18} remain valid if we assume the following condition instead of A\ref{asssq}.

\begin{assumption}\label{assexp}
There exists a real-valued finite measure~$\mm$
on~$(\xx, \bb)$ with the following property: if a measurable function $g:\xx\to\rr$ is such that for all $\lambda\in\rr,\ \lambda>0$ holds
$\int_{\xx}2^{\lambda |g|}\,d\mm<+\infty$ then $g$ is integrable w.~r.~t.~$\mu$ on~$\xx$.
\end{assumption}

Obviously, condition A\ref{assexp} is weaker than A\ref{asssq}.

Our proofs are very similar to respective proofs in~\cite{radt09} and~\cite{radt18} with one essential change -- now we use Lemma~\ref{lm:fkmu} from the present paper. In the previous publications, the statement of Lemma~\ref{lm:fkmu} was obtained under condition A\ref{asssq}, in our paper, we assume A\ref{assexp}. Note that in many of our results, we assume that the paths of our processes are continuous.

The rest of the paper is organised as follows. In Section~\ref{sc:prel} we recall the basic facts concerning SMs and Besov spaces and give a short literature review. In Section~\ref{sc:auxlem} we prove Lemma~\ref{lm:fkmu} and other auxiliary lemmas. Section~\ref{sc:besovr} contains the result about the Besov regularity of SMs on $[0,1]^d$. In Section~\ref{sc:four} we study the Fourier series defined by process $\mu(t)=\mu((0,t])$, $t\in[0,2\pi]$.

\section{Preliminaries}\label{sc:prel}

\subsection{Stochastic measures}\label{ss:stochm}

In this subsection, we give basic information concerning stochastic measures in a general setting. In statements of Sections~\ref{sc:besovr} and \ref{sc:four}, this set function is defined on Borel subsets of $[0,1]^d$ or $[0,2\pi]$.

Let $\llf_0=\llf_0(\Omega, {\ff}, {\pp})$ be the set of all real-valued
random variables defined on the complete probability space $(\Omega, {\ff}, {\pp} )$. Convergence in $\llf_0$ means the convergence in probability. Let ${\xx}$ be an arbitrary set and ${\bb}$ a $\sigma$-algebra of subsets of ${\xx}$.

\begin{defin}\label{def:stme}
A $\sigma$-additive mapping $\mu:\ {\bb}\to \llf_0$ is called {\em stochastic measure} (SM).
\end{defin}

We do not assume the moment existence or martingale properties for SM. We can say that $\mu$ is an $\llf_0$--valued measure.

We note the following examples of SMs. The orthogonally scattered stochastic measures are SMs with values in ${\llf}_2(\Omega,\ff,\pp)$. The $\alpha$-stable random measures defined on a $\sigma$-algebra for $\alpha\in (0,1)\cup(1,2]$ are independently scattered SMs, see \cite[Chapter 3]{samtaq}, Definition~\ref{def:stme} holds by \cite[Proposition 3.5.1]{samtaq}.

Many examples of the SMs on the Borel subsets of $[0,T]$ may be given by the Wiener-type integral $\mu({\aaa})=\int_{[0,T]} {\mathbf 1}_{\aaa}(t)\,dX_t$.
For example, this holds if $X_t$ be any square integrable martingale or fractional Brownian motion with Hurst index $H>1/2$. Other examples may be found in~\cite[Section 2.1]{manrad24} or~\cite[Section 1.2.1]{radbook}.

For deterministic measurable functions $f:\xx\to\rr$, an integral of the form $\int_{\xx}f\,d\mu$ is studied
in~\cite[Chapter 7]{kwawoy}, ~\cite[Chapter 1]{radbook}. In particular, every bounded measurable $f$ is integrable with respect to (w.r.t.) any~$\mu$, and for any integrable $g$ holds
\begin{equation}\label{eq:limf}
\lim_{c\to +\infty}\sup_{|h|\le |g|}\pp\Bigl\{\Bigl|\int_{\xx}h\,d\mu\Bigr|\ge c\Bigr\} = 0
\end{equation}
(this follows from Corollary 1.1, Lemma~1.8, and Theorem~1.2 ~\cite{radbook}). An analogue of the Lebesgue dominated
convergence theorem holds for this integral (see~\cite[Proposition 7.1.1]{kwawoy} or~\cite[Theorem 1.5]{radbook}).

The theory of SMs in detail is considered in~\cite{radbook}. Equations driven by SMs are studied, for example, in \cite{bodnarchuk20}, \cite{manikin24}, \cite{shen20}. SMs may be used for the study of stochastic dynamical systems (see~ \cite{bai22}, \cite{baiowwa20}).

In the sequel, $\bb(\xx)$ denote the Borel $\sigma$-algebra of subsets of $\xx$, $\mm_L$ denote the Lebesgue measure.

\subsection{Besov spaces}\label{ss:besovsp}

These classical functional Banach spaces found many applications in mathematical physics, function theory, and functional analysis. We recall the definition of Besov spaces following~\cite{kamont}.

For functions $f\in \llf_p([0, 1]^d)=\llf_p([0, 1]^d,\mm_L)$ we set
\begin{equation*}
\|f\|_{B^\alpha_{p,q}([0, 1]^d)}=\|f\|_{\llf_{p}([0, 1]^d)}+\Bigl(\int_0^{1} (\omega_p(f, r))^q r^{-\alpha q-1}
\,dr\Bigr)^{1/q},
\end{equation*}
where $\omega_p$ denotes the $\llf_p$-modulus of continuity,
\begin{eqnarray*}
\omega_p(f, r)  =  \sup_{|h|\le r}\Bigl(\int_{I_{h}} |f(x+h)-f(x)|^p\,dx\Bigr)^{1/p},\\
I_{h}  =  \{x\in [0, 1]^d:\ x+h\in [0, 1]^d \},
\end{eqnarray*}
$|h|$ denotes the Euclidean norm in $\rr^d$. Then
\[
B^\alpha_{p,q}([0, 1]^d)=\{f\in \llf_p([0, 1]^d):\ \|f\|_{B^\alpha_{p,q}([0, 1]^d)}<+\infty\},
\]
and $\|\cdot\|_{B^\alpha_{p,q}([0, 1]^d)}$ is a norm in this space.

Besov regularity of trajectories of random processes is well-studied, see, for example, \cite{cikero}, \cite{fageot17}, \cite{herren}. For Besov--Orlicz spaces, similar results are obtained in \cite{ondver20}, \cite{coupond24}. Besov regularity and continuity of paths of multidimensional integral w.r.t. SMs are considered in~\cite{manrad24}.

\section{Auxiliary lemmas}\label{sc:auxlem}

In the first statement, we recall the well-known Paley--Zygmund inequality (see, for example, Lemma~4.3(a)~\cite{vahtar} or Lemma~2.1~\cite{radbook}).

\begin{lemma}\label{lm:pz}
Let $\varepsilon_{k},\ 1\le k\le m,$ be independent Bernoulli random variables, \[
\pp[\varepsilon_{k}=1]=\pp[\varepsilon_{k}=-1]=1/2.
\]
Then for each $\lambda_k\in\rr$
\begin{equation}\label{eq:pzcr}
\pp\Bigl[\Bigr(\sum_{k=1}^{m} \lambda_{k} \varepsilon_{k}\Bigr)^2\ge\frac{1}{4}\sum_{k=1}^{m} \lambda_{k}^2
\Bigr]\ge\frac{1}{8}.
\end{equation}
\end{lemma}

The following lemma is the main result of this section. Assuming A\ref{asssq} a similar statement was obtained in Lemma~3.3~\cite{radt09}.

\begin{lemma} \label{lm:fkmu} Assume that $\mu$ be an SM on $(\xx,\bb)$ and Assumption~A\ref{assexp} holds.

Let the measurable functions $f_k:\ {\xx}\to \rr,\ k\ge 1$, be such that
\begin{equation} \label{eq:sumf}
\sup_{x\in\xx}\sum_{k=1}^{\infty} ({f_k}(x))^2<\infty.
\end{equation}
Then
\begin{equation}
\label{eqinfk} \sum_{k=1}^{\infty} \Bigl(\int_\xx f_k\,d\mu\Bigr)^2< +\infty\quad\mbox{\textrm ~a.~s.}
\end{equation}
\end{lemma}

\begin{proof}
Consider independent Bernoulli random variables $\varepsilon_{k},\ k\ge 1,$ defined on some other probability
space $(\Omega', \ff',\pp')$,
\[
\pp'[\varepsilon_{k}(\omega')=1]=\pp'[\varepsilon_{k}(\omega')=-1]=1/2.
\]

Consider the sum
\begin{equation}
\label{eqetxo} \eta(x, \omega')=\sum_{k=1}^{\infty} \varepsilon_{k}(\omega') f_k(x).
\end{equation}

The well-known two--series theorem and condition~\eqref{eq:sumf} imply that for each such~$x$ series~\eqref{eqetxo} converges $\pp'$-a.~s. on~$\Omega'$.

Denote
\[
C_f=\sup_{x\in\xx}\sum_{k=1}^{\infty} ({f_k}(x))^2.
\]
Then
\begin{eqnarray*}
\pp'\{\lambda|\eta(x,\omega')|\ge 1\}\le {\textrm{Var}}\,(\lambda|\eta(x,\omega')|)=\lambda^2 \sum_{k=1}^{\infty} f^2_k(x)\le \lambda^2 C_f.
\end{eqnarray*}
Take $\lambda_0>0$ such that $\lambda_0^2 C_f\le 1/64$.

Theorem 2.3 of~\cite{kahane} states the following fact for each fixed $x$. If
\[
\pp'\{\lambda|\eta(x,\omega')|\ge r\}\le a/2
\]
then
\[
\pp'\{\lambda|\eta(x,\omega')|\ge 2r\}\le a^2.
\]
We have that
\[
\pp'\{\lambda_0|\eta(x,\omega')|\ge 1\}\le 2^{-6},
\]
therefore
\[
\pp'\{\lambda_0|\eta(x,\omega')|\ge 2^k\}\le 2^{-2^{k+2}-2}.
\]
The L\'{e}vy inequality for symmetric variables (see, for example, \cite[Lemma~2.3]{kahane}) implies that for any $x\in\xx,\ c>0$ for
\[
\zeta(x,\omega')=\sup_{n\ge 1}\Bigl|\sum_{k=1}^{n} \varepsilon_{k}(\omega') f_k(x)\Bigr|
\]
holds
\begin{equation*}
\pp'[\zeta(x, \omega')> c]\le 2\pp'[|\eta(x, \omega')|> c].
\end{equation*}
Therefore,
\[
\pp'\{\lambda_0 \zeta(x,\omega') \ge 2^k\}\le 2^{-2^{k+2}-1} .
\]

We get
\begin{eqnarray*}
\ee_{\pp'}2^{\lambda_0 \zeta(x,\omega') }\le 2\pp'\{{\lambda_0 \zeta(x,\omega') }\le 2\}
+\sum_{k=1}^{\infty} 2^{2^{k+1}}\pp'\{2^k\le {\lambda_0 \zeta(x,\omega') }\le 2^{k+1}\}\\
\le 2+ \sum_{k=1}^{\infty} 2^{2^{k+1}}\pp'\{{\lambda_0 \zeta(x,\omega') }\ge 2^{k}\}\le 2+ \sum_{k=1}^{\infty} 2^{2^{k+1}} 2^{-2^{k+2}-1}<3.
\end{eqnarray*}

By the Fubini -- Tonelli theorem,
\[
\ee_{\pp'}\int_\xx2^{\lambda_0 zeta(x,\omega')}\,d\mm(x)=\int_\xx\ee_{\pp'}2^{\lambda_0 \zeta(x,\omega') }\,d\mm(x)\le 3\mm(\xx)<+\infty.
\]

Therefore,
\begin{equation*}
\int_\xx 2^{\lambda_0 \zeta(x,\omega') }\,d\mm(x) <+\infty\quad \pp'-\textrm{a.~s.}
\end{equation*}
and Assumption~A\ref{assexp} implies that $\zeta(x, \omega')$ is integrable w.~r.~t.~$\mu$ $\pp'$-a.~s.

For all $j\ge 1$
\[
\Bigl|\sum_{k=1}^{j} \varepsilon_{k}(\omega') f_k(x)\Bigr|\le \zeta(x, \omega').
\]
For each $\omega'$ (excluding a set of zero $\pp'$-measure) we use the dominated convergence theorem \cite[Theorem 1.5]{radbook} for the integral w.~r.~t.~$\mu$ and obtain
\begin{equation*}
\int_\xx \eta(x, \omega') \,d\mu(x)=\sum_{k=1}^{\infty} \varepsilon_{k}(\omega')\int_\xx
f_k(x)\,d\mu(x),
\end{equation*}
where the last series converges in probability~$\pp$. Using this, it is easy to obtain that this series converges in probability~$\pp\times\pp'$ on~$\Omega\times\Omega'$, and its partial sums are bounded in probability~$\pp\times\pp'$.

Further, define $\xi_k(\omega)=\int_{\xx} f_k\,d\mu$. Suppose \eqref{eqinfk} fails.  Then
\begin{equation}\label{eq:exde}
\exists \delta_{0}>0\ \forall c>0\ \exists m_c\ge 1:\ \pp\Bigl[\sum_{k=1}^{m_c}
\xi_k^2(\omega)\ge c\Bigr]\ge\delta_{0} \, .
\end{equation}

Applying~\eqref{eq:pzcr} for $\lambda_{k}=\xi_k(\omega)$, and \eqref{eq:exde} obtain
\[
\pp'\Bigl[\omega':\
\Bigl(\sum_{k=1}^{m_c} \varepsilon_{k}(\omega')\xi_k(\omega)\Bigr)^2\ge \frac{c}{4}\Bigr]\ge\frac{1}{8}
\]
for each fixed $\omega\in{\Omega}_c=\Bigl\{\sum_{k=1}^{m_c}\xi_k^2\ge c\Bigr\}$, and $\pp({\Omega}_c)\ge\delta_{0}$. Integrating over the set~${\Omega}_c$, we get
\begin{equation*}
\pp\times\pp'\Bigl[(\omega, \omega'):\ \Bigl(\sum_{k=1}^{m_c}
\varepsilon_{k}(\omega')\xi_k(\omega)\Bigr)^2\ge\frac{c}{4}\Bigr] \ge\frac{\delta_0}{8}\, .
\end{equation*}

Thus, there exist $\omega_0'$ such that
\[
\pp\Bigl[\omega:\ \Bigl(\sum_{k=1}^{m_c} \varepsilon_{k}(\omega_0') \xi_k(\omega)
\Bigr)^2\ge\frac{c}{4}\Bigr]\ge\frac{\delta_{0}}{8}\,
\]
and $\zeta(x, \omega_0')$ is integrable w.r.t. $\mu$.
For the function
$ {g}(x)=\sum_{k=1}^{m_c} \varepsilon_{k}(\omega_0')f_k(x)$ we have
\begin{equation*}
| {g}(x)|\le \zeta(x, \omega_0'),\quad \pp\Bigl[\Bigl|\int_{\xx}{g}\,d\mu\Bigr|\ge
\frac{\sqrt{c}}{2}\Bigr]\ge\frac{\delta_{0}}{8}\, .
\end{equation*}
Recall that $\delta_{0}>0$ is fixed and $c$ is arbitrary. Therefore, we obtain a contradiction with \eqref{eq:limf}.
\end{proof}

\begin{example} We give an example, where we can check Assumption~A\ref{assexp}, and can not check~A\ref{asssq}.

Let $\varepsilon_{k},\ k\ge 1$, be independent Bernoulli random variables. Consider an SM on Borel subsets on $[0,1]$
\[
\mu(\aaa)=\sum_{k=1}^\infty \dfrac{\varepsilon_{k}}{k^{4/3}}\int_{\aaa}x^{k^{-1/3}-1}\,dx.
\]
This series converges a.s. for each $\aaa\in\bb((0,1])$ because
\[
\sum_{k=1}^\infty \Bigl(\dfrac{1}{k^{4/3}}\int_{\aaa}x^{k^{-1/3}-1}\,dx\Bigr)^2<\infty,
\]
and $\mu$ is an SM with values in $\llf_2(\Omega,\ff,\pp)$.

For any measurable function $f:[0,1]\to\rr$ such that
\begin{equation*}
\sum_{k=1}^\infty \Bigl(\dfrac{1}{k^{4/3}}\int_{[0,1]}|f(x)|x^{k^{-1/3}-1}\,dx\Bigr)^2<\infty
\end{equation*}
$f$ is integrable w.r.t. $\mu$, and holds
\begin{equation}\label{eq:serf}
\int_{\aaa} f\,d\mu = \sum_{k=1}^\infty \dfrac{\varepsilon_{k}}{k^{4/3}}\int_{\aaa}f(x)x^{k^{-1/3}-1}\,dx.
\end{equation}
This fact is obvious for simple function $f$. For arbitrary measurable $f$, we take simple $f_n$ that  converge to $f$ pointwise, $|f_n(x)|\le |f(x)|$. Then
\begin{eqnarray*}
\sum_{k=1}^\infty \Bigl(\dfrac{1}{k^{4/3}}\int_{\aaa}(f(x)-f_n(x))x^{k^{-1/3}-1}\,dx\Bigr)^2\to 0,\quad n\to\infty,\\
\Rightarrow \sum_{k=1}^\infty \dfrac{\varepsilon_{k}}{k^{4/3}}\int_{\aaa}f_n(x)x^{k^{-1/3}-1}\,dx\stackrel{\llf_2}{\to}
\sum_{k=1}^\infty \dfrac{\varepsilon_{k}}{k^{4/3}}\int_{\aaa}f(x)x^{k^{-1/3}-1}\,dx,\quad n\to\infty
\end{eqnarray*}
for each $\aaa\in\bb([0,1])$. Theorem~1.8~\cite{radbook} implies that $f$ is integrable w.r.t. $\mu$, and \eqref{eq:serf} holds.

For measurable $f$ such that $\int_{[0,1]}2^{\lambda |f(x)|} \,dx<\infty$, we will obtain that \eqref{eq:serf} fulfills. Applying the H\"{o}lder inequality get
\begin{equation}\label{eq:hold}
\begin{split}
\int_{[0,1]}|f(x)|x^{k^{-1/3}-1}\,dx
\le \Bigl(\int_{[0,1]}|f(x)|^{2k^{1/3}-1}\,dx\Bigr)^{\frac{1}{2k^{1/3}-1} } \Bigl(\int_{[0,1]}x^{(1/2)k^{-1/3}-1}\,dx\Bigr)^{\frac{2k^{1/3}-2}{2k^{1/3}-1} }\\
\le 2k^{1/3} \Bigl(\int_{[0,1]}|f(x)|^{2k^{1/3}-1}\,dx\Bigr)^{\frac{1}{2k^{1/3}-1}}.
\end{split}
\end{equation}
It is easy to calculate that
\begin{equation}\label{eq:diff}
|f(x)|^{2k^{1/3}-1}\le C_{k,\lambda} 2^{\lambda |f|},\quad {\textrm{where}}\quad
C_{k,\lambda}=\Bigl(\frac{2k^{1/3}-1}{\lambda\ln 2}\Bigr)^{2k^{1/3}-1} 2^{-\frac{2k^{1/3}-1}{\ln 2}}.
\end{equation}
We get
\begin{eqnarray*}
\sum_{k=1}^\infty \Bigl(\dfrac{1}{k^{4/3}}\int_{\aaa}|f(x)|x^{k^{-1/3}-1}\,dx\Bigr)^2
\stackrel{\eqref{eq:hold}}{\le} 4 \sum_{k=1}^\infty \dfrac{1}{k^{2}} \Bigl(\int_{[0,1]}|f(x)|^{2k^{1/3}-1}\,dx\Bigr)^{\frac{2}{2k^{1/3}-1}}\\
\stackrel{\eqref{eq:diff}}{\le} C \sum_{k=1}^\infty \dfrac{1}{k^{2}} \Bigl(\int_{[0,1]}\Bigl(\frac{2k^{1/3}-1}{\lambda\ln 2}\Bigr)^{2k^{1/3}-1} 2^{-\frac{2k^{1/3}-1}{\ln 2}} 2^{\lambda |f|} \,dx\Bigr)^{\frac{2}{2k^{1/3}-1}}\\
\le C \sum_{k=1}^\infty \dfrac{1}{k^{4/3}}  \Bigl(\int_{[0,1]} 2^{\lambda |f|} \,dx\Bigr)^{\frac{2}{2k^{1/3}-1}}<\infty,
\end{eqnarray*}
where $C$ denotes a positive constant. Therefore, $f$ is integrable w.r.t. $\mu$.

Thus, Assumption~A\ref{assexp} holds for $\mm=\mm_L$. At the same time, it is not clear how we can check A\ref{asssq} for our~$\mu$. For the Lebesgue measure $\mm_L$ or $\mm(\aaa)=\int_{\aaa} x^\alpha\,dx$, $-1<\alpha<0$ we can find $f(x)=x^\beta$, $\beta<0$, such that $f^2$ is integrable to $\mm$ but some elements of sum in~\eqref{eq:serf} are not defined. Investigation of all finite $\mm$ on $\bb([0,1])$ looks difficult.
\hfill$\square$
\end{example}

The following statements will be used in Section~\ref{sc:four} for the series expansion of SMs.

\begin{lemma}\label{lm:sumf}
Let Assumption~A\ref{assexp} holds. Then the set of random variables
\[
\Bigl\{\sum_{k=1}^{j} \Bigl(\int_{\xx} f_k\,d\mu\Bigr)^2\ \Bigr|\ f_k:{\xx}\to{\rr}\ {\rm are\ measurable},\ \sum_{k=1}^{j} f_k^2(x)\le 1,\ j\ge 1\Bigr\}
\]
is bounded in probability.
\end{lemma}

\begin{proof} If the statement fails, for some $\delta_{0}>0$ and all $n\ge 1$ we can find functions $f_{kn}$, $1\le k\le j_n$, such that
\[
\sum_{k=1}^{j_n} f_{kn}^2(x)\le 1,\quad {\pp}\Bigl\{\sum_{k=1}^{j_n} \Bigl(\int_{\xx} f_{kn}\,d\mu\Bigr)^2>2^n\Bigr\}>\delta_{0}.
\]
Then
\[
\sum_{n=1}^{\infty}\sum_{k=1}^{j_n} \bigl(2^{-n/2}f_{kn}(x)\bigr)^2\le 1,\quad \sum_{n=1}^{\infty}\sum_{k=1}^{j_n} \Bigl(\int_{\xx} 2^{-n/2}f_{kn}\,d\mu\Bigr)^2\ {\rm does\ not\ converge},
\]
which contradicts Lemma~\ref{lm:fkmu}. \end{proof}

\begin{lemma}\label{lminmd}
Let Assumption~A\ref{assexp} holds, $\mu$ be an SM on $\bb([0,T])$, and the process $\mu(t)=\mu((0,t])$, $0\le t\le T$ have continuous paths. Then for any $T_1$, $0<T_1<T$ we have
\[
\int_{[0,T_1]}\frac{\bigl|\mu(s+\varepsilon)-\mu(s)\bigr|^{3}}{\varepsilon}\,ds\stackrel{\pp}{\to} 0,\quad \varepsilon\to 0+.
\]
\end{lemma}

\begin{proof}
We have that
\begin{equation}\label{eqinmu}
\int_{[0,T_1]}\frac{\bigl|\mu(s+\varepsilon)-\mu(s)\bigr|^{3}}{\varepsilon}\,ds\le
\sup_s\bigl|\mu(s+\varepsilon)-\mu(s)\bigr| \int_{[0,T_1]}\frac{\bigl|\mu(s+\varepsilon)-\mu(s)\bigr|^{2}}{\varepsilon}\,ds.
\end{equation}

For any $n\ge 1$ take the partition of $[0,T_1]$ by points $s_{kn}=\Bigl(\frac{k}{n}T_1\varepsilon\Bigr)\wedge T_1$, $0\le k\le j_n$, and consider the Riemann integral sum for the last integral in~\eqref{eqinmu}
\begin{equation}\label{eqriems}
\begin{split}
\sum_{k=1}^{j_n} \frac{\bigl|\mu(s_{kn}+\varepsilon)-\mu(s_{kn})\bigr|^{2}}{\varepsilon}\,\frac{T_1\varepsilon}{n}
 =T_1\sum_{k=1}^{j_n} \Bigl(\int_{[0,T_1]} f_{kn}\,d\mu\Bigr)^2,\\
 {\rm where}  \quad f_{kn}(x)=\frac{1}{\sqrt{n}}\,{\bf 1}_{(s_{kn},s_{kn}+\varepsilon]}(x).
 \end{split}
\end{equation}
We have $\sum_{k=1}^{j_n} f_{kn}^2(x)\le 1$, by Lemma~\ref{lm:sumf} set of sums~\eqref{eqriems} is bounded in probability. For continuous $\mu(s)$, $\sup_s\bigl|\mu(s+\varepsilon)-\mu(s)\bigr|\to 0$ as $\varepsilon\to 0$, and this implies the statement of Lemma~\ref{lminmd}.
\end{proof}

\section{Besov regularity of SM in $[0,1]^d$}\label{sc:besovr}

Now we will consider SM $\mu$ defined on the Borel $\sigma$-algebra of~$[0, 1]^d$, $d\ge 1$ and obtain the Besov regularity of $\mu$ with continuous realizations.

For $x=(x_1, x_2, \dots, x_d)\in[0, 1]^d$ set
\[
\mu(x)=\mu\Bigl(\prod_{i=1}^{d}[0, x_i]\Bigr).
\]

By $e^{(i)}$ we denote the $i$th coordinate unit vector in~$\rr^d$, and consider discrete sets
\begin{eqnarray*}
U(n, i)=\Bigl\{y= \Bigl( \dfrac{k_1}{2^n}, \dfrac{k_2}{2^n}, \dots, \dfrac{k_d}{2^n}\Bigr)\ \Big|\
k_j=0,\ 1,\ \dots,\ 2^n,\
 1\le j\le d; \ y+2^{-n} e^{(i)} \in [0, 1]^d \Bigr\}.
\end{eqnarray*}
By Corollary~3.3~\cite{kamont}, if for each $i$, $1\le i\le d$, for respective fixed $\omega\in\Omega$ it holds
\begin{equation}
\label{eqsudu2} \sum_{n=1}^{\infty}2^{n(\alpha p-d)}\sum_{y\in U(n, i)} |\mu(y + 2^{-n}e^{(i)})-\mu(y) |^p<+\infty.
\end{equation}
then continuous path of $\mu(x)$ belong to $B^\alpha_{p,p}([0, 1]^d)$. Applying this result we prove the following statement.

\begin{thm}
\label{thnbspd} Let Assumption~A\ref{assexp} holds, and the random function $\mu(x),\ x\in[0,1]^d$, have continuous realizations. Then for any $1\le p<+\infty,\ 0<\alpha<\min\{1/p,1/2\}$ the realization $\mu(x)$, $x\in[0,1]^d$ with probability 1 belongs to the Besov space~$B^\alpha_{p,p}([0,1]^d)$.
\end{thm}

\begin{proof}
Firstly, consider the case  $2\le p<+\infty$, then \mbox{$0<\alpha<1/p$}. To obtain~\eqref{eqsudu2}, it is sufficient, for each~$i$, to prove the convergence of the series
\begin{equation}\label{eq:srvy2}
\begin{split}
\sum_{n=1}^{\infty}2^{ n (\alpha p-d)}\sum_{y\in U(n,i)} |\mu (y + 2^{-n} e^{(i)})- \mu (y) |^2
=\sum_{n\ge 1,y\in U(n,i)} \Bigl(2^{n (\alpha p-d)/2}\int_{[0,1]^d}h_{n, y} (x)\,d\mu (x)\Bigr)^2,
\end{split}
\end{equation}
where
\[
h_{n, y}(x)=\ii_{\{x_j\le{k_j}/{2^n},\ j\ne i,\ {k_i}/{2^n}<x_i\le{(k_i+1)}/{2^n}\}}(x),\quad y= \Bigl( \dfrac{k_1}{2^n},\ \dfrac{k_2}{2^n},\ \dots,\ \dfrac{k_d}{2^n}\Bigr).
\]

Here, for each $n$, each fixed~$x$ belongs to at most $2^{n(d-1)}$ sets from the indicators in the functions~$h_{n,y}$. Therefore,
\begin{equation*}
\begin{split}
\sum_{n\ge 1,y\in U(n,i)} (2^{ n (\alpha p-d)/2}h_{n, y}(x))^2
 \le  \sum_{n=1}^{\infty} 2^{ n (\alpha p-d)} 2^{n(d-1)}
 =  \sum_{n=1}^{\infty} 2^{ n (\alpha p-1)}<+\infty.
 \end{split}
\end{equation*}
Lemma~\ref{lm:fkmu} implies that series~\eqref{eq:srvy2} converge a.~s.

For $1\le p<2$, we can repeat the proof of this case from Theorem~5.1~\cite{radt09} (or Theorem~2.2~\cite{radbook}), and obtain that \eqref{eq:srvy2} coverge for $\alpha=1/2$. \end{proof}

Note that, for paths of multidimensional integral w.r.t. SM, sufficient conditions of Besov regularity are given in~\cite[Theorem~5]{manrad24}. Now we see that we can change A\ref{asssq} to A\ref{assexp} in that statement because we can refer to our Theorem~\ref{thnbspd} instead of  Theorem~5.1~\cite{radt09} in the proof of Theorem~5~\cite{manrad24}.

\section{Fourier expansion of SM}\label{sc:four}

Let $\mu$ be  an SM on $\bb([0,2\pi])$. Consider the random process $\mu(t)=\mu((0,t])$, $0\le t\le 2\pi$.

Assume that paths of the process $\mu(t)$ are Riemann integrable on $[0,2\pi]$. Consider the Fourier series defined by $\mu(t)$ for each fixed $\omega\in\Omega$:
\begin{eqnarray}
\xi_k= \frac{1}{\pi} \int_{[0,2\pi]}\mu(s) \cos k s \,ds,\quad
\eta_k= \frac{1}{\pi} \int_{[0,2\pi]}\mu(s) \sin k s \,ds,\label{gaxiet}\\
\mu(t)\sim \frac{\xi_0}{2}+\sum_{k=1}^{\infty} (\xi_k\cos kt+\eta_k\sin kt).\label{eq:series}
\end{eqnarray}
Set
\begin{equation*}
S_n(t)=\frac{\xi_0}{2}+\sum_{k=1}^n (\xi_k\cos kt+\eta_k\sin kt).
\end{equation*}

Applying the integration by parts formula for integrals in~\eqref{gaxiet} (see~\cite[Lemma~1]{radt18} or~\cite[Lemma 2.7]{radbook}), we get
\begin{equation}\label{eq:dfex}
\begin{split}
\xi_k  = -\frac{1}{k\pi} \int_{(0,2\pi]}\sin ks \,d\mu,\quad \eta_k=\frac{1}{k\pi}\int_{(0,2\pi]}(\cos ks-1)
\,d\mu,\quad k\ge 1,\\
\xi_0  = 2\mu((0,2\pi])-\frac{1}{\pi} \int_{(0,2\pi]}s \,d\mu,
\end{split}
\end{equation}
and these integrals are defined for any SM $\mu$ without the assuming Riemann integrability.

In the sequel, we will assume that  $\xi_k$ and $\eta_k$ are defined by~\eqref{eq:dfex}, an consider the Fourier series for arbitrary SM $\mu$ on~$\bb([0,2\pi])$. First, we prove that series \eqref{eq:series} converge.

\begin{thm} \label{thcoas} If Assumption~A\ref{assexp} holds then
\[
\pp\bigl[S_n(t)\ \textrm{converges }\mm_L\textrm{-a.s.\ on\ }[0,2\pi]\bigr]=1.
\]
\end{thm}

\begin{proof}
Applying Lemma~\ref{lm:fkmu} for
\[
f_k=\dfrac{\sin kt}{\pi k}\quad \textrm{and}\quad f_k=\dfrac{\cos kt-1}{\pi k},
\]
using~\eqref{eq:dfex}, we get $\sum_k(\xi_k^2+\eta_k^2)<\infty$ a.~s. Applying the famous Carleson's theorem (see~\cite{carles}) we obtain our statement.
\end{proof}

Further, we prove that, under some assumptions, $S_n$ converge to values of $\mu$.

Representations of processes in the form of random series began with the well-known Paley--Wiener expansion of the Wiener process. A similar representation of the fractional Brownian motion was obtained in~\cite{dzhzan}. We give such a result for continuous paths of SMs.

The following statement gives a generalization of the Dirichlet -- Jordan theorem about Fourier series expansion of functions of bounded variation, see, for example, \cite[Theorem II.8.1]{zygmund}.

\begin{thm}
\label{th:convs}
Let Assumption~A\ref{assexp} holds, and paths of the process $\mu(t)$, $0\le t\le 2\pi$, be continuous. Then for any $t\in (0,2\pi)$ it holds $S_n(t)\rp\mu(t)$ and
$S_n(0)=S_n(2\pi)\rp\mu(2\pi)/2$ as $n\to\infty$.
\end{thm}

\begin{proof}
Without loss of generality, we may assume that $\mu((0,2\pi])=0$, and the periodic continuation of
$\mu(t)$ to~$\rr$ is continuous. Otherwise, we can consider the SM
\[
\widetilde{\mu}(\aaa)=\mu(\aaa)-\mm_L(\aaa)\dfrac{\mu((0,2\pi])}{2\pi}.
\]

Set
\begin{eqnarray*}
S_n^*(t) = \frac{1}{2}(S_{n-1}(t)+S_n(t))=S_{n-1}(t)+\frac{1}{2}(\xi_n\cos nt+\eta_n\sin nt),\\
\varphi_t(s) = \frac{1}{2}(\mu(t+s)+\mu(t-s)-2\mu(t))=\frac{1}{2}\bigl(\mu((t,t+s])-\mu((t-s,t])\bigr).
\end{eqnarray*}
From Theorem~(II.10.1)~\cite{zygmund} it follows that for $h=\pi / n$ it holds
\begin{eqnarray*}
|S_n^*(t)-\mu(t)| \le \frac{1}{\pi}\int_{h}^{\pi}\frac{|\varphi_t(s)-\varphi_t(s+h)|}{s}\,ds
+h \int_{h}^{\pi}\frac{|\varphi_t(s)|}{s^2}\,ds
 + \frac{2}{h} \int_0^{2h}{|\varphi_t(s)|}\,ds+o(1)\\
:=I_1+I_2+I_3+o(1).
\end{eqnarray*}
Continuity of $\mu(t)$ and L'H\^{o}pital's rule give that $I_2, I_3\to 0$ as $h\to 0$ for each $t\in\rr$ and $\omega\in\Omega$. Using the H\"{o}lder inequality, we obtain
\begin{eqnarray*}
 2\pi I_1 = \int_{h}^{\pi}\frac{|\mu((t+s,t+s+h])-\mu((t-s-h,t-s])|}{s}\,ds\\
 \le  \Bigl( \int_{h}^{\pi}\frac{h^{1/2}}{s^{3/2}}\,ds\Bigr)^{\frac{2}{3}} \Bigl(\int_{h}^{\pi}\frac{|\mu((t+s,t+s+h])-\mu((t-s-h,t-s])|^3}{h}\,ds\Bigr)^{\frac{1}{3}}.
\end{eqnarray*}
The last value tends to zero in probability as $h\to 0$ by Lemma~\ref{lminmd}.

Thus, $S_n^*(t)\rp\mu(t)$. By the analogue of the Lebesgue dominated convergence theorem (see~\cite[Proposition 7.1.1]{kwawoy} or~\cite[Theorem 1.5]{radbook}),
 $\xi_n,\ \eta_n\rp 0$, and therefore $S_n(t)\rp\mu(t)$.
\end{proof}

\begin{remark} Also, Lemma~\ref{lm:sumf} of our paper admits to change Assupmtion~A\ref{asssq} to Assumption~A\ref{assexp} in Theorem~3.1~2)~\cite{radavsm19}, where, for an equation driven by SM, the rate of the convergence in the averaging principle is established. This is obvious from the proof of that theorem.
\end{remark}


\bibliographystyle{bib/vmsta-mathphys}
\bibliography{RadchenkoRegulSM24arxiv}

\end{document}